\documentclass{article}

\usepackage{amsmath, amsfonts, amssymb}
\usepackage{graphicx}
\usepackage{hyperref}

\newcommand{\F}{\mathbb{F}_q}

\newcommand{\Ev}{\mathrm{Ev}}

\newcommand{\FF}{\mathcal{F}}
\newenvironment{proof}[1][Proof]{\textbf{#1.} }{\ \rule{0.5em}{0.5em}}
\newcommand{\qed}{\hfill \ensuremath{\Box}}

\newtheorem{theorem}{Theorem}

\newtheorem{corollary}[theorem]{Corollary}

\newtheorem{definition}[theorem]{Definition}
\newtheorem{remark}[theorem]{Remark}
\newtheorem{lemma}[theorem]{Lemma}
\newtheorem{proposition}[theorem]{Proposition}
\newtheorem{example}[theorem]{Example}

\begin{document}

\title{Bounding the number of points on a curve using a generalization of Weierstrass semigroups\thanks{This work was supported in part by the Danish FNU grant 272-07-0266, the Danish National Research Foundation and the National Science Foundation of China (Grant No.11061130539) for the Danish-Chinese Center for Applications of Algebraic Geometry in Coding Theory and Cryptography and by Spanish grant MTM2007-64704.}}%

\author{Peter Beelen\footnote{DTU-Mathematics, Technical University of Denmark, Matematiktorvet, Building 303, 2800 Kgs. Lyngby, Denmark \texttt{P.Beelen@mat.dtu.dk}} \and Diego Ruano\footnote{Department of Mathematical Sciences, Aalborg University, Fr. Bajersvej 7G, 9220 Aalborg {\O}st, Denmark. \texttt{diego@math.aau.dk}}}

\maketitle   

\begin{abstract}
In this article we use techniques from coding theory to derive upper bounds for the number of rational places of the function field of an algebraic curve defined over a finite field. The used techniques yield upper bounds if the (generalized) Weierstrass semigroup \cite{BeTu} for an $n$-tuple of places is known, even if the exact defining equation of the curve is not known. As shown in examples, this sometimes enables one to get an upper bound for the number of rational places for families of function fields. Our results extend results in \cite{GeMa}.
\end{abstract}

\section{Introduction}

Let $\F$ be the finite field with $q$ elements and $\FF / \F$ be a function field \cite{St} of an algebraic curve $\mathcal C$ defined over $\F$. We denote by $N(\FF)$ the number of rational places of $\FF$ and by $g(\FF)$ its genus. Even when the defining equation of $\mathcal C$ is known explicitly, it can be useful to have a priori upper bounds for $N(\FF)$. If only partial information is available about the curve $\mathcal C$, it is often still possible to give an upper bound on the number of rational places of $\FF$. One such upper bound is the well-known Hasse--Weil upper bound, stating that $N(\FF) \le q+1+2g(\FF)\sqrt{q}$. To use this upper bound, one only needs to know (an upper bound for) the genus of $\FF$. In \cite[Theorem 1]{GeMa} another type of an a priori upper bound is given, which assumes the knowledge of the Weierstrass semigroup $H(P_1)$ of a rational place $P_1$ of $\FF$: $$N(\FF) \le \# \left ( H(P_1) \setminus (q H^*(P_1) + H(P_1))  \right) +1,$$
where  $q H^*(P_1) + H(P_1) = \{ q \lambda + \lambda' ~ | ~ \lambda, \lambda' \in H(P_1), \lambda \neq 0 \}$.
One may rightly ask how often the situation arises in which one does not know the exact equation of $\mathcal C$, but one does know a Weierstrass semigroup. However, we will show in examples that having only some information on the defining equation sometimes is enough to compute the bound in \cite{GeMa} as well as our generalized bounds (see Example \ref{ex:second}). In order to extend the Geil--Matsumoto bound, we will in Section \ref{se:ggm} consider the Weierstrass semigroup defined by several rational places \cite{BeTu}. In section \ref{se:seco}, we estimate the size of certain subsets of the set of rational places. This second estimation can lead to a sharper estimate of the total number of rational places. As done in \cite{GeMa}, one may change viewpoint and use the bounds obtained in this article to obtain information about the kind of (generalized) semigroups that may occur when one assumes that the function field has many rational places. This is also the point of view taken in Example \ref{ex:second}, where it is shown that a certain family of function fields of genus $6$ cannot improve upon known records from \cite{ManyPoints}.

The main techniques to prove our results come from coding theory. More precisely, we consider AG-codes constructed by evaluating functions from a Riemann-Roch space $L(G)$ (for suitable divisors $G$) in rational places of $\FF$. The length of the resulting code is given by the number of rational places used in this construction. Usually, the rational places are fixed and one is interested in determining the minimum distance of the code. In this article, modifying an idea from \cite{GeMa}, we estimate the dimension of the code. Since the dimension of a code cannot exceed its length, this gives information about the number of rational places the function field $\FF$ can have.

\section{A generalization of the Geil--Matsumoto bound}\label{se:ggm}

In this section we will present a first generalization of the Geil--Matsumoto bound \cite{GeMa}. Let $P_1,\dots,P_n$ be $n$ rational places of a given function field $\mathcal F$ and denote by $\mathcal Q$  the set of $N({\mathcal F})-n$ remaining rational places. Note though that in the next section, $\mathcal Q$ will in general denote a subset of these $N({\mathcal F})-n$ places. For an $n$-tuple ${\bf i}=(i_1,\dots,i_n) \in \mathbb{Z}^n$ we write $\deg({\bf i})=\sum_{j=1}^n i_j$ and $L({\bf i})=L(\sum_{j=1}^n i_jP_j)$. Further we will denote with ${\bf e}_j$ the $n$-tuple all of whose coordinates are $0$, except the $j$-th one, which is assumed to be $1$. Then one has for example that $L(\lambda{\bf e}_j)=L(\lambda P_j)$.

\begin{definition}\label{def:hi}
Given ${\bf i} \in \mathbb{Z}^n$, we define $$H_{\bf i}(P_j)=\{-v_{P_j}(f) \quad | \quad f \in \cup_{k \in \mathbb{Z}}L({\bf i}+k{{\bf e}_j})\backslash\{0\}\}$$
\end{definition}

\begin{remark}
\begin{enumerate}
\item We have $H_{\bf 0}(P_j)=H(P_j)$, where ${\bf 0}$ denotes the $n$-tuple consisting of zeroes only.
\item Note that the set $H_{\bf i}(P_j)$ does not depend on the $j$-th coordinate of $\bf i$.
\item We remark that  $L({\bf i}+k{\bf e}_j)=\{0\}$, for $k < -\deg({\bf i})$, so it also holds that $$H_{\bf i}(P_j)=\{-v_{P_j}(f) \quad | \quad f \in \cup_{k \ge -\deg({\bf i})}L({\bf i}+k{\bf e}_j)\backslash\{0\}\}.$$
\item Sets such as $H_{\bf i}(P_j)$ were also mentioned in \cite{genorderbound}, where they were used to compute lower bounds on the minimum distances of certain algebraic geometry codes. In \cite{genorderbound} it is also explained how to compute these sets. They are closely related to the generalized Weierstrass semigroups introduced in \cite{BeTu}.

\end{enumerate}
\end{remark}

With this notation in place, we define the following functions:

\begin{definition}\label{def:negligible}
Let ${\bf i}\in \mathbb{Z}^n$ and let $j$ be an integer between $1$ and $n$. If $L({\bf i}) = L({\bf i}+{\bf e}_j)$ or if there exists $\lambda \in H(P_j)\backslash\{0\}$ and $\mu \in H_{\bf i}(P_j)$ such that $\mu + q\lambda = {\bf i}_j+1$, we call the pair $({\bf i},{\bf i}+{\bf e}_j)$ negligible. Further we define
$$
\delta({\bf i},{\bf i}+{\bf e}_j)=\left\{\begin{array}{rcl}
0  & & \makebox{ if the pair $({\bf i},{\bf i}+{\bf e}_j)$ is negligible,}\\
1  & & \makebox{ otherwise. }
\end{array}\right.
$$
\end{definition}

\begin{lemma}\label{lem:reduce}
Let $({\bf i},{\bf i}+{\bf e}_j)$ be a negligible pair such that $L({\bf i}) \subsetneq L({\bf i}+{\bf e}_j)$, and write $\mu + q\lambda = {\bf i}_j+1$ for $\lambda \in H(P_j)\backslash\{0\}$ and $\mu \in H_{\bf i}(P_j)$. Then there exist $f \in L(\lambda {\bf e}_j)$ and $g \in L({\bf i})$ such that $f^q g \in L({\bf i}+{\bf e}_j)\backslash L({\bf i})$.
\end{lemma}
\begin{proof}
Since $\lambda \in H(P_j)$, there exists a function $f \in L({\bf e}_j)$ whose pole divisor equals $(f)_\infty=\lambda P_j$. Similarly there exists a function $g \in L({\bf i})$ such that $(g) \ge -\sum_{k=0}^n i_k P_k$ and $-v_{P_k}(g)=\mu$. This implies that $-v_{P_j}(f^qg)=q\lambda+\mu={\bf i}_j+1$ and also that $(f^q g) \ge -P_j-\sum_{k=0}^n i_k P_k$. Together these imply that $f^qg \in L({\bf i}+{\bf e}_j)\backslash L({\bf i})$ as desired. \qed
\end{proof}

A pair $({\bf i},{\bf i}+{\bf e}_j)$ is negligible if $\deg({\bf i})$ is large enough. More precisely, one has:

\begin{proposition}\label{prop:finite}
Let ${\bf i}\in \mathbb{Z}^n$ and let $j$ be an integer between $1$ and $n$ and assume that $\deg({\bf i}) \ge (q+2)(g(\mathcal F)+1)-3$. Then the pair $({\bf i},{\bf i}+{\bf e}_j)$ is negligible.
\end{proposition}
\begin{proof}
Suppose that $\deg({\bf i}) \ge (q+2)(g(\mathcal F)+1)-3$. Since this implies in particular that $\deg({\bf i})\ge 2g(\mathcal F)-1$, the theorem of Riemann--Roch implies that $L({\bf i}) \subsetneq L({\bf i}+{\bf e}_j)$. Since the semigroup $H(P_j)=\{0,\lambda,\dots\}$ has exactly $g(\mathcal F)$ gaps, there exists an integer $\lambda \in H(P_j)\backslash\{0\}$ with $\lambda \le g(\mathcal F)+1$. Therefore $\deg({\bf i}+(1 - q \lambda){\bf e}_j) \ge 2g(\mathcal F)$, so applying the theorem of Riemann--Roch again, we see that there exists a function $g \in L({\bf i}+(1 - q \lambda){\bf e}_j)$ such that $-v_{P_j}(g)={\bf i}_j+1-q\lambda$. By Definition \ref{def:hi}, we see that ${\bf i}_j+1-q\lambda \in H_{\bf i}(P_j)$. By Definition \ref{def:negligible} the proposition now follows, since $({\bf i}_j+1-q\lambda)+q\lambda={\bf i}_j+1$. \qed
\end{proof}

Actually we showed the following more precise result:

\begin{corollary}\label{cor:howfar}
Let $\lambda_j$ denote the smallest nonzero element of $H(P_j)$. Then the pair $({\bf i},{\bf i}+{\bf e}_j)$ is negligible if $\deg({\bf i}) \ge q \lambda_j+2g(\mathcal F)-1$.
\end{corollary}

Now we come to the main theorem.

\begin{theorem}\label{thm:main}
Define $M=(q+2)(g(\mathcal F)+1)-3$ and let ${\bf i}^{(-1)},\dots,{\bf i}^{(M)}$ be a sequence of $n$-tuples such that:
\begin{enumerate}
\item $\deg({\bf i}^{(-1)})=-1$,
\item for any $k$ there exists a $j$ such that ${\bf i}^{(k)}-{\bf i}^{(k-1)}={\bf e}_j$.
\end{enumerate}
Then $N(\mathcal F) \le n+\sum_{k=0}^{M} \delta({\bf i}^{(k-1)},{\bf i}^{(k)})$.
\end{theorem}
\begin{proof}
Note that by the properties of the divisor sequence, we have $\deg({\bf i}^{(k)})=k$ for any $-1 \le k \le M$. For any divisor $G$ with support disjoint from $\mathcal Q$, we introduce the following notation:
$$
\begin{array}{rcl}
\Ev_{\mathcal Q}: L(G) & \to      & \F^{N(\mathcal F)-n}\\
\\
                    f  &  \mapsto & (f(Q))_{Q\in {\mathcal Q}}
\end{array}
$$
and $C_{\mathcal Q}(G)=\Ev_{\mathcal Q}(L(G)).$ For an $n$-tuple ${\bf i}$, we define $$C_{\mathcal Q}({\bf i})=\Ev_{\mathcal Q}(L({\bf i})).$$

We will begin the proof of the theorem by showing the following three claims:
\begin{enumerate}
\item For any divisor $G$ of degree $\deg(G) \ge N(\mathcal F)-n+2g({\mathcal F})-1$, it holds that $C_{\mathcal Q}(G)=\F^{N({\mathcal F})-n}.$
\item For any $k\ge 0$ we have $\dim(C_{\mathcal Q}({\bf i}^{(k)})) \le \dim(C_{\mathcal Q}({\bf i}^{(k-1)}))+\delta({\bf i}^{(k-1)},{\bf i}^{(k)})$.
\item $\dim(C_{\mathcal Q}({\bf i}^{(-1)}))=0.$
\end{enumerate}

The first claim follows from a standard argument: the kernel of the evaluation map $\Ev_{\mathcal Q}: L(G) \to \F^{N(\mathcal F)-n}$ is given by $L(G-\sum_{Q\in{\mathcal Q}}Q)$. Therefore we get $\dim(C_{\mathcal Q}(G))=\dim(L(G))-\dim(L(G-\sum_{Q\in{\mathcal Q}}Q))$. Using the assumption $\deg(G) \ge N(\mathcal F)-n+2g({\mathcal F})-1$ and the theorem of Riemann--Roch, this expression simplifies to $N({\mathcal F})-n$.

The second claim is trivial if $\delta({\bf i}^{(k-1)},{\bf i}^{(k)})=1$, so we may assume that $\delta({\bf i}^{(k-1)},{\bf i}^{(k)})=0$. Since by assumption there exists $j$ such that ${\bf i}^{(k)}={\bf i}^{(k-1)}+{\bf e}_j$, we may apply Lemma \ref{lem:reduce} to conclude that there exist $f \in L(\lambda {\bf e}_j)$ for some $\lambda>0$ and $g \in L({\bf i}^{(k-1)})$ such that $f^q g \in L({\bf i}^{(k)})\backslash L({\bf i}^{(k-1)})$. On the level of codes this means that the code $C_{\mathcal Q}({\bf i}^{(k)})$ is generated as a vector space by the vectors of $C_{\mathcal Q}({\bf i}^{(k-1)})$ and the vector $\Ev_{\mathcal Q}(f^q g)$. However, since the codes are defined over $\F$, we have $\Ev_{\mathcal Q}(f^q g)=\Ev_{\mathcal Q}(f g)$. On the other hand, since $\lambda>0$, we see that $fg \in L({\bf i}^{(k-1)})$ and therefore that $\Ev_{\mathcal Q}(f g) \in C_{\mathcal Q}({\bf i}^{(k-1)})$. The second claim now follows.

The third claim is clear, since $L(G)=\{0\}$ for any divisor of negative degree.

\bigskip
\noindent
From the last two parts of the claim we find inductively that $$\dim( C_{\mathcal Q}({\bf i}^{(M)})) \le \sum_{k=0}^{M} \delta({\bf i}^{(k-1)},{\bf i}^{(k)}).$$ On the other hand, combining a similar reasoning and Proposition \ref{prop:finite}, we find that $$\dim( C_{\mathcal Q}({\bf i}^{(M)}))=\dim( C_{\mathcal Q}({\bf i}^{(M)}+l{\bf e}_j))$$ for any $j$ and any natural number $l$. From this and the first claim we can conclude that $$\dim( C_{\mathcal Q}({\bf i}^{(M)}))=N({\mathcal F})-n.$$ The theorem now follows. \qed
\end{proof}

\noindent
The above proof is inspired by the proof of the Geil--Matsumoto bound \cite{GeMa}. If $n=1$, the above theorem reduces to their result: If $n=1$, the only choice for the sequence ${\bf i}^{(-1)},\dots,{\bf i}^{(M)}$ is $-1,0,\dots,M$. For $n>1$, there are more possibilities. In fact, we obtain a weighted oriented graph given by the lattice with vertices $\{-1, \ldots , M \}^n$ and edges $( {\bf i} , {\bf i}+{\bf e}_j )$, with weights $w( {\bf i} , {\bf i}+{\bf e}_j ) =  \delta({\bf i},{\bf i}+{\bf e}_j)$, for ${\bf i} \in \{-1, \ldots , M \}^n$ and $j=1, \ldots , n$ such that $i_j \neq M$. In practice, we consider the bound from Corollary \ref{cor:howfar} instead of the bound $M$ in Theorem \ref{thm:main}. We do not need to consider the whole lattice, but can start with a one-dimensional lattice and increase its size progressively. Then we just find an optimal sequence ${\bf i}^{(-1)},\dots,{\bf i}^{(M)}$, by finding a path from a vertex with degree $-1$ to a vertex with degree $M$ with minimum weight (using Dijkstra's algorithm \cite{Dijkstra}, which computes a path with lowest weight between a particular vertex of a graph and every other vertex of that graph).

\begin{example} \label{ex:klein}
In this example we consider the function field of the Klein quartic over $\mathbb{F}_8$ which has genus three. It can be described as $\FF_1/ \mathbb{F}_8=\mathbb{F}_8(x,y) / \mathbb{F}_8$, where $x^3 y + y^3 + x = 0$. Of course it is well-known how many rational places this function field has (namely $24$) and it should be noted that the only purpose of this example is to illustrate the theory.

There are three rational places occurring as poles and/or zeroes of the functions $x$ and $y$. We will denote these by $P_1$, $P_2$ and $P_3$. More precisely we choose them such that the following identities of divisors hold: $$(x)= 3 P_1 -P_2 -2P_3 \makebox{ and } (y) = P_1 + 2P_2 -3P_3.$$  From this, one can show that $H=H(P_1)=H(P_2)=H(P_3)=\langle 3, 5, 7 \rangle$ and
\begin{equation}\label{eq:LspaceKlein}
L(i_1 P_1 + i_2 P_2 + i_3 P_3) = \langle x^{\alpha} y^{\beta}  ~ | ~ 3\alpha + \beta \ge - i_1, -\alpha + 2 \beta \ge -i_2, -2 \alpha -3 \beta \ge - i_3  \rangle.
\end{equation}
Actually, one can prove that all rational places of the Klein quartic have the same Weierstrass semigroup. The Geil--Matsumoto bound gives $N(\FF_1) \le 1+ 24= 25$ in this case, since $H \setminus (8 H^* + H) = \{ 0, 3, 5, 6, \ldots, 23, 25, 26, 28 \}$.

We now compute the bound from Theorem \ref{thm:main}, where we will consider $n=2$, and $P_1$, $P_2$ as above. It is enough to consider a sequence of $n$-tuples $({\bf i}^{(-1)},\dots,{\bf i}^{(29)})$, since $({\bf i},{\bf i}+{\bf e}_j)$ is negligible if $\deg({\bf i}) \ge 8\cdot 3 +2 \cdot 3-1=29$ (Corollary \ref{cor:howfar}). As before we represent the divisor $P_1$, resp. $P_2$ by ${\bf e}_1$, resp. ${\bf e}_2$ and write $${\bf i}^k = (i_1^{(k)} , i_2^{(k)} ) = i_1 ^{(k)} {\bf e}_1 + i_2^{(k)} {\bf e}_2.$$

We computed a oriented graph as above, given by the $\{-1,\ldots 29 \} \times \{0,\ldots 4 \}$ lattice, with weights given by $\delta({\bf i},{\bf i}+{\bf e}_j)$ and got a path with minimum weight given by

$$
\left\{ \begin{array}{lcl}
{\bf i}^{(k)}=(k,0), & & \makebox{for $k=-1,\ldots, 23$,}\\
{\bf i}^{(23+k)}=(24,k-1),  & & \makebox{for $k=1,\ldots, 3$,}\\
{\bf i}^{(26+k)}=(25,k+1), & & \makebox{for $k=1,\ldots, 3$.}\\
\end{array} \right.
$$
Then $\{k\ge 0 ~ | ~ \delta({\bf i}^{(k-1)},{\bf i}^{(k)})=1 \} = \{0,3,5,6,\ldots,23,25 \}$, which implies that $N(\mathcal F) \le 2 + 22 =24$.
\end{example}

The Geil--Matsumoto bound is an improvement to the gonality bound, sometimes called Lewittes' bound \cite{Le}, $$N(\FF) \le q \lambda_1 +1,$$where $\lambda_1$ denotes the smallest non-zero element of $H$. In the above example these bounds give rise to the same upper bound for $N(\FF)$. The following proposition explains this phenomenon. We introduce the Ap\'ery set of a numerical semigroup \cite{Ap,RoGa}, which is the main tool for this result. For $e \in H$, the Ap\'ery set of $H$ relative to $e$ is defined to be $\mathrm{Ap}(H,e) = \{ \lambda \in H | \lambda - e \notin H\}$. One has that $\mathrm{Ap}(H,e)$ is $\{w_0=0, w_1, \ldots, w_{e-1}  \}$, where $w_i$ is the smallest element of $H$ congruent with $i$ modulo $e$, for $i=0,\ldots ,e-1$. Moreover, for $\lambda \in H$ there exist  a unique $i$ and $k$, with $i \in \{0 , \ldots, e-1 \}$ and $k \in \mathbb{N}_0$, such that $\lambda = w_i + k e$. Thus we have the disjoint union $$H = \bigcup_{i=0}^{e-1} \{ w_i + e\mathbb{N}_0  \},$$ in particular $\{e,  w_1, \ldots, w_{e-1} \}$ generates $H$.

\begin{proposition}\label{prop:GM=L}
Let $e \in H$ and $\lambda_1$ the smallest non-zero element of $H$, then $$\#(H \setminus (e H^* + H )) =e\lambda_1.$$ In particular the bounds in \cite{GeMa,Le} give the same result if $q \in H$.
\end{proposition}

\begin{proof}
Let $\mathrm{Ap}(H,e) = \{ w_0=0, w_1, \ldots, w_{e-1} \}$ be the Ap\'ery set of $H$ relative to $e \in H$. For any $\lambda \in H$ we have
$$e\lambda+H=\bigcup_{i=0}^{e-1}(e\lambda+w_i+e\mathbb{N}_0)=\bigcup_{i=0}^{e-1}(w_i+e\mathbb{N}_{\ge \lambda}),$$ where $\mathbb{N}_{\ge \lambda}$ denotes the set of natural numbers greater than or equal to $\lambda$. This implies that for $\lambda<\mu \in H$ we have $e\lambda+H \supset e\mu+H$ and thus
$$e H^* + H=e\lambda_1+H,$$ with $\lambda_1$ the smallest element of $H^*$. The proposition now follows, since the equality $\#(H \setminus ( e \lambda_1 + H))=e\lambda_1$ is a well-known result for semigroups, see \cite[Chapter 10, Lemma 5.15]{HoLiPe}.
\qed
\end{proof}

The Weierstrass semigroup of Example \ref{ex:klein} contains $q=8$, the number of elements of the base field. Therefore, both bounds in \cite{GeMa,Le} give the same result. Namely, we have $e=q=8$ and $w_0=0$, $w_1=9$, $w_2=10$, $w_3=3$, $w_4=12$, $w_5=5$, $w_6=6$, and $w_7=7$.

\begin{remark}\label{rem:converseGM=L}
The converse statement, namely that (in the notation of Proposition \ref{prop:GM=L}) $\#(H \setminus (e H^* + H )) =e\lambda_1$ implies that $e\in H$, is not necessarily true. Consider for example the semigroup $\{0,2,4,5,6,\dots\}$ generated by $2$ and $5$ and suppose that $e=3$.

On the other hand, for semigroups $H$ generated by $m$ and $m+1$, with $m$ a natural number, this converse does hold: Suppose that $\#(H \setminus (e H^* + H )) =em$, then $e(m+1) \in em+H$ (by \cite[Chapter 10, Lemma 5.15]{HoLiPe}), which would imply that $e=e(m+1)-em\in H$.
\end{remark}

We conclude this section with an example that shows that the above techniques also can be used when only partial information is given about the defining equation of the function field.

\begin{example}\label{ex:second}
In this example we consider the function field $\FF_2/ \mathbb{F}_8=\mathbb{F}_8(x,y) / \mathbb{F}_8$, where $x$ and $y$ are related by an equation of the form
\begin{equation}\label{eq:defeq}
\alpha y^4+\beta y+x^5 + \sum_{(i,j)\in\Delta}a_{i,j}x^iy^j = 0,
\end{equation}
where $\alpha$ and $\beta$ are nonzero elements in $\mathbb{F}_8$ and $a_{i,j}$ are arbitrary elements in $\mathbb{F}_8$. Moreover $\Delta=\{(i,j)\in \mathbb{Z}^2 ~ | ~ 4i+5j<20,\ i \ge 0, \ 5j+i > 5\}$. Another way of stating the structure of the defining equation is that its Newton polygon is a triangle with vertices $(0,0)$, $(5,0)$ and $(0,4)$. We will assume that the genus of $\FF_2$ equals $6$, which amounts to saying if we interpret Equation (\ref{eq:defeq}) as a defining equation of a curve, then this curve does not have any singularities.

Like for the Klein quartic in the previous example, we can derive information about the divisors for $x$ and $y$ from the defining equation. In fact, denoting by $P_1$ the common zero of $x$ and $y$ and by $P_2$ the unique pole of $x$ (and $y$), we have $$(x)\ge P_1 - 4 P_2 \makebox{ and } (y)=5(P_1-P_2).$$ Using the assumption that $g(\FF_2)=6$, one can show that this implies
\begin{equation}\label{eq:Lspacesecondexample}
L(i_1 P_1 + i_2 P_2) = \langle x^{\alpha} y^{\beta}  ~ | ~ \alpha + 5 \beta \ge - i_1, -4\alpha -5 \beta \ge -i_2, i_1\ge 0 \rangle.
\end{equation}
In particular, we see that the semigroup of $P_2$ (and in fact also $P_1$) is generated by $4$ and $5$. Since $8 \in H(P_1)$, we see that the bound from \cite{GeMa} gives $33$ for this example. Any function field of genus $6$ defined over $\mathbb{F}_8$ can have at most $34$ points (i.e. $N_8(6) \le 34$), while it is also known that $N_8(6) \ge 33$ \cite{ManyPoints}. Based on this, one may hope that for a clever choice of the coefficients $\alpha,\beta$ and the $a_{i,j}$ one can find a function field defined by an equation of the form as in (\ref{eq:defeq}) with $33$ rational places. However, it turns out that using Theorem \ref{thm:main} with
$$
\left\{ \begin{array}{lcl}
{\bf i}^{(k)}=(k,0), & & \makebox{for $k=-1,\ldots, 34$,}\\
{\bf i}^{(34+k)}=(34,k),  & & \makebox{for $k=1,\ldots, 3$,}\\
{\bf i}^{(37+k)}=(34+k,3), & & \makebox{for $k=1,\ldots, 3$,}\\
{\bf i}^{(40+k)}=(37,3+k), & & \makebox{for $k=1,\ldots, 3$.}\\
\end{array} \right.
$$
that $N(\FF_2) \le 2 + 29 =31$. Note that we do not have to describe more values of ${\bf i}^{(k)}$ by Corollary \ref{cor:howfar}.
\end{example}

\section{A second generalization of the Geil--Matsumoto bound}\label{se:seco}

In this section we will generalize the previous results by estimating the size of certain subsets of the set of rational places. Contrary to the previous section, we will therefore in this section by $\mathcal Q$ denote some subset of the set of all rational places not containing any of the places $P_1,\dots,P_n$. The results from the previous section can be refined in this setup. Further we define $T=\mathbb{F}_q \backslash \{0\}$ for convenience.

\begin{definition}\label{def:Tnegligible}
Let ${\bf i}\in \mathbb{Z}^n$ and let $j$ be an integer between $1$ and $n$. We call the pair $({\bf i},{\bf i}+{\bf e}_j)$ $T$-negligible if either $L({\bf i}) = L({\bf i}+{\bf e}_j)$ or if

\begin{enumerate}
\item there exists $\lambda \in H(P_j)\backslash\{0\}$ and $\mu \in H_{\bf i}(P_j)$ such that $\mu + (q-1)\lambda = {\bf i}_j+1$ and
\item for this $\lambda$ there exists $f \in L(\lambda P_j)\backslash L((\lambda-1)P_j)$ such that $f(Q) \in T$ for all $Q \in {\mathcal Q}$.
\end{enumerate}

Further we define
$$
\delta_T({\bf i},{\bf i}+{\bf e}_j)=\left\{\begin{array}{rcl}
0  & & \makebox{ if the pair $({\bf i},{\bf i}+{\bf e}_j)$ is $T$-negligible,}\\
1  & & \makebox{ otherwise. }
\end{array}\right.
$$
\end{definition}

Note that depending on the choice of $\mathcal Q$, the function $\delta_T$ may change. Strictly speaking we should therefore include $\mathcal Q$ in the notation for this function, but for the sake of simplicity, we will not do this.

\begin{lemma}\label{lem:Treduce}
Let $({\bf i},{\bf i}+{\bf e}_j)$ be a $T$-negligible pair such that $L({\bf i}) \subsetneq L({\bf i}+{\bf e}_j)$ and write $\mu + (q-1)\lambda = {\bf i}_j+1$ for $\lambda \in H(P_j)\backslash\{0\}$ and $\mu \in H_{\bf i}(P_j)$. Then there exist $f \in L(\lambda {\bf e}_j)$ and $g \in L({\bf i})$ such that $f^{q-1} g \in L({\bf i}+{\bf e}_j)\backslash L({\bf i})$ and such that moreover $f(Q) \in T$ for all $Q \in {\mathcal Q}$.
\end{lemma}
\begin{proof}
Since $\lambda \in H(P_j)$, there exists a function $f \in L({\bf e}_j)$ whose pole divisor equals $(f)_\infty=\lambda P_j$. By Definition \ref{def:Tnegligible} we can choose a function $f$ such that $f(Q) \in T$ for all $Q \in {\mathcal Q}$. Similarly there exists a function $g \in L({\bf i})$ such that $(g) \ge -\sum_{k=0}^n i_k P_k$ and $-v_{P_j}(g)=\mu$. This implies that $-v_{P_j}(f^{q-1}g)=(q-1)\lambda+\mu={\bf i}_j+1$ and also that $(f^{q-1} g) \ge -(q-1)\lambda P_j-\sum_{j=0}^n i_j P_j$. Together these imply that $f^{q-1}g \in L({\bf i}+{\bf e}_j)\backslash L({\bf i})$ as desired.\qed
\end{proof}

A pair $({\bf i},{\bf i}+{\bf e}_j)$ is negligible if $\deg({\bf i})$ is large enough. More precisely, one has:

\begin{proposition}\label{prop:Tfinite}
Let ${\bf i}\in \mathbb{Z}^n$ and let $j$ be an integer between $1$ and $n$. Define $\Lambda=\#{\mathcal Q}+2g({\mathcal F})$ and $M_T=(q-1)\Lambda+2g(\mathcal F)-1$. Then any pair $({\bf i},{\bf i}+{\bf e}_j)$ satisfying $\deg({\bf i}) \ge M_T$ is $T$-negligible.
\end{proposition}
\begin{proof}
Suppose that $\deg({\bf i}) \ge M_T.$ Since then in particular $\deg({\bf i})\ge 2g(\mathcal F)-1$, it follows from the theorem of Riemann--Roch that $L({\bf i}) \subsetneq L({\bf i}+{\bf e}_j)$.  Also note that $\deg({\bf i}+(1 - (q-1)\Lambda){\bf e}_j) \ge 2g(\mathcal F)$, so applying the theorem of Riemann--Roch again, we see that there exists a function $g \in L({\bf i}+(1 -(q-1)\Lambda){\bf e}_j)$ such that $-v_{P_j}(g)={\bf i}_j+1-(q-1)\Lambda$.
By Definition \ref{def:hi}, we see that ${\bf i}_j+1-(q-1)\Lambda \in H_{\bf i}(P_j)$.

Since the largest gap of the semigroup $H(P_j)$ is at most $2g({\mathcal F})-1$, the number $\Lambda$ is not a gap of $H(P_j)$. This means that there exists a function $f \in L(\Lambda P_j)$ such that $-v_{P_j}(f)=\Lambda$. We cannot conclude yet from Definition \ref{def:Tnegligible} that the pair $({\bf i},{\bf i}+{\bf e}_j)$ is $T$-negligible, since $f$ could have a zero among the places in $\mathcal Q$. However, from the proof of Theorem \ref{thm:main} and the definition of $\Lambda$ we see that for any $j$ the evaluation map $\Ev_{\mathcal Q}: L((\Lambda-1) P_j) \to \F^{\#{\mathcal Q}}$ is surjective. Therefore, we can always choose $f \in L(\Lambda P_j)\backslash L((\Lambda-1) P_j)$ such that $f(Q)\in T$ for all $Q\in{\mathcal Q}$.\qed
\end{proof}

The $M_T$ given in this proposition can be very large. Under some additional conditions, we can obtain better results.

\begin{proposition}\label{prop:Tfinite2}
Let ${\bf i}\in \mathbb{Z}^n$ and let $j$ be an integer between $1$ and $n$. Suppose that for any $\lambda \in H(P_j)$ there exists $f \in L(\lambda P_j)\backslash L((\lambda-1)P_j)$ such that $f(Q) \in T$ for all $Q \in {\mathcal Q}$. If $\deg({\bf i}) \ge (q+1)(g(\mathcal F)+1)-3$, then the pair $({\bf i},{\bf i}+{\bf e}_j)$ is $T$-negligible.
\end{proposition}
\begin{proof}
Suppose that $\deg({\bf i}) \ge (q+1)(g(\mathcal F)+1)-3$. Since then $\deg({\bf i})\ge 2g(\mathcal F)-1$, it follows from the theorem of Riemann--Roch that $L({\bf i}) \subsetneq L({\bf i}+{\bf e}_j)$. As in the proof of Proposition \ref{prop:finite} we can conclude that there exists $\lambda \in H(P_j)\backslash\{0\}$ with $\lambda \le g(\mathcal F)+1$. This implies that $\deg({\bf i}+(1 - (q-1)\lambda){\bf e}_j) \ge 2g(\mathcal F)$, so applying the theorem of Riemann--Roch again, we see that there exists a function $g \in L({\bf i}+(1 - (q-1)\lambda){\bf e}_j)$ such that $-v_{P_j}(g)={\bf i}_j+1-(q-1)\lambda$. By Definition \ref{def:hi}, we see that ${\bf i}_j+1-(q-1)\lambda \in H_{\bf i}(P_j)$. Furthermore by assumption, there exists $f \in L(\lambda P_j)\backslash L((\lambda-1)P_j)$ such that $f(Q) \in T$ for all $Q \in {\mathcal Q}$. Therefore, by Definition \ref{def:Tnegligible}, the proposition follows.\qed
\end{proof}

As in the previous section, we can refine the above statement:

\begin{corollary}\label{cor:Thowfar}
Let $\lambda_j$ denote the smallest nonzero element of $H(P_j)$. Suppose that for any $\lambda \in H(P_j)$ there exists $f \in L(\lambda P_j)\backslash L((\lambda-1)P_j)$ such that $f(Q) \in T$ for all $Q \in {\mathcal Q}$. Then the pair $({\bf i},{\bf i}+{\bf e}_j)$ is $T$-negligible if $\deg({\bf i}) \ge (q-1)\lambda_j+2g(\mathcal F)-1$.
\end{corollary}

Now we come to the refinement of Theorem \ref{thm:main}.

\begin{theorem}\label{thm:Tmain}
Define $\Lambda=\#{\mathcal Q}+2g({\mathcal F})$ and $M_T=(q-1)\Lambda+2g(\mathcal F)-1$. Let ${\bf i}^{(-1)},\dots,{\bf i}^{(M_T)}$ be a sequence of $n$-tuples such that:
\begin{enumerate}
\item $\deg({\bf i}^{(-1)})=-1$,
\item for any $k$ there exists a $j$ such that ${\bf i}^{(k)}-{\bf i}^{(k-1)}={\bf e}_j$.
\end{enumerate}
Then $\#{\mathcal Q} \le \sum_{k=0}^{M_T} \delta_T({\bf i}^{(k-1)},{\bf i}^{(k)})$.
\end{theorem}
\begin{proof}
The proof is similar to that of Theorem \ref{thm:main}. All the reasoning is similar apart from the proof of the following claim: 

\bigskip
\noindent
For any $k\ge 0$ we have $\dim(C_{\mathcal Q}({\bf i}^{(k)})) \le \dim(C_{\mathcal Q}({\bf i}^{(k-1)}))+\delta_T({\bf i}^{(k-1)},{\bf i}^{(k)})$.

\bigskip
\noindent
This is clear if $\delta_T({\bf i}^{(k-1)},{\bf i}^{(k)})=1$, so we may assume that $\delta_T({\bf i}^{(k-1)},{\bf i}^{(k)})=0$. We may apply Lemma \ref{lem:Treduce} to conclude that there exist $f \in L(\lambda {\bf e}_j)$ for some $\lambda>0$ and $g \in L({\bf i}^{(k-1)})$ such that
$f^{q-1} g \in L({\bf i}^{(k)})\backslash L({\bf i}^{(k-1)})$. Moreover, we may assume that $f(Q) \in T$ for all $Q \in {\mathcal Q}.$ Since $\alpha^{q-1}=1$ for all $\alpha \in T$, this implies $f(Q)^{q-1}=1$ for all $Q \in {\mathcal Q}$. On the level of codes we have, as in Theorem \ref{thm:main}, that the code $C_{\mathcal Q}({\bf i}^{(k)})$ is generated as a vector space by the vectors of $C_{\mathcal Q}({\bf i}^{(k-1)})$ and the vector $\Ev_{\mathcal Q}(f^{q-1} g)$. However, we have $\Ev_{\mathcal Q}(f^{q-1} g)=\Ev_{\mathcal Q}(g) \in C_{\mathcal Q}({\bf i}^{(k-1)})$. The claim now follows and the proof of the theorem can be concluded as that of Theorem \ref{thm:main}.\qed
\end{proof}

In case $n=1$ and the hypotheses from Proposition \ref{prop:Tfinite2} are satisfied, we obtain the following result:

\begin{corollary}\label{cor:Tnis1}
Suppose that for any $\lambda \in H(P_1)$ there exists $f \in L(\lambda P_1)\backslash L((\lambda-1)P_1)$ such that $f(Q) \in T$ for all $Q \in {\mathcal Q}$. Then
$$\#{\mathcal Q} \le \#H(P_1)\backslash ((q-1)H^*(P_1)+H(P_1)).$$
\end{corollary}
\begin{proof}
Since $n=1$, the only sequence we can choose is $-1,0,1,2,\dots$. However, under the stated assumptions, a pair $(k-1,k)$ is $T$-negligible if and only if $k \in (q-1)H^*(P_1)+H(P_1)$.\qed
\end{proof}

We will now give some examples.

\begin{example}
This example is a continuation of Example \ref{ex:klein}. In particular we will use the same notation as in that example. We choose $P=P_1$ and $\mathcal Q$ to be the set of all rational places $Q$ satisfying $x(Q) \in T$ and $y(Q) \in T$. Using the divisors for $x$ and $y$ in Example \ref{ex:klein}, we see that the only rational places not in $\mathcal Q$ are $P_1$, $P_2$ and $P_3$.

Using Equation (\ref{eq:LspaceKlein}), we see that the conditions in Corollary \ref{cor:Tnis1} are satisfied for our choice of $\mathcal Q$. Therefore we find that

$$\#{\mathcal Q} \le \# H(P_1)\backslash ( 7H^*(P_1)+H(P_1) )=\#\{0,3,5,\dots,20,22,23,25\}=21.$$
Also counting the rational points $P_1$, $P_2$ and $P_3$ we find that $N(\mathcal F_1) \le 24$. Since $N(\mathcal F_1) = 24$, Corollary \ref{cor:Tnis1} gives a tight bound in this case.
\end{example}

\begin{example}
This example is a continuation of Example \ref{ex:second}. There we have seen that $N(\FF_2)\le 31$. This can also been seen using Corollary \ref{cor:Tnis1}. From Equation (\ref{eq:defeq}), we see that there are at most $4$ rational places $P$ of $\FF_2$ with $x(P)=0$ or $y(P)=0$ and a unique pole of $x$ and $y$. On the other hand applying Corollary \ref{cor:Tnis1} to the semigroup $H(P_2)=\langle 4,5 \rangle$, we get that $\#\mathcal{Q} \le 26$. Therefore $N(\FF_2) \le 4+1+26=31$.
\end{example}

\end{document}